\documentclass{amsart}

\usepackage[utf8x]{inputenc}
\usepackage{amssymb,latexsym}
\usepackage{amsmath}
\usepackage{graphicx}
\usepackage{amsthm}
\usepackage[shortlabels]{enumitem}
\usepackage{appendix}
\usepackage{comment}
\usepackage[foot]{amsaddr}

\newtheorem{theorem}{Theorem}[section]
\newtheorem{lemma}[theorem]{Lemma}
\newtheorem{proposition}[theorem]{Proposition}
\newtheorem{conjecture}[theorem]{Conjecture}
\newtheorem{corollary}[theorem]{Corollary}

\newtheorem{claim}{Claim}[theorem]

\newtheorem{notation}[theorem]{Notation}
\newtheorem{definition}[theorem]{Definition}

\newtheorem{question}[theorem]{Question}
\theoremstyle{remark}
\newtheorem{remark}[theorem]{Remark}
\newtheorem{example}[theorem]{Example}
\newtheorem{examples}[theorem]{Examples}

\newcommand{\R}{\mathbb{R}}
\newcommand{\M}{\mathcal{M}}

\newcommand{\ch}{\mathcal{H}}

\renewcommand{\geq}{\geqslant}

\renewcommand{\ge}{\geqslant}
\renewcommand{\le}{\leqslant}

\def\P{\mathbb{P}}
\def\L{\mathbb{L}}
\def\U{\mathbb{U}}

\def\V{\mathbb{V}}

\def\b1{\bar{1}}
\def\cb1{\cdot \bar{1}}

\usepackage{array,color,colortbl}

\newcommand{\cP}{\mathbb{P}}

\newcommand{\cH}{\mathcal{H}}

\newcommand{\refT}[1]{Theorem~\ref{#1}}
\newcommand{\refC}[1]{Corollary~\ref{#1}}
\newcommand{\refCl}[1]{Claim~\ref{#1}}
\newcommand{\refL}[1]{Lemma~\ref{#1}}

\newcommand{\refS}[1]{Section~\ref{#1}}

\newcommand{\refE}[1]{Example~\ref{#1}}

\newcommand{\refCon}[1]{Conjecture~\ref{#1}}

\newcommand{\cplus}{\boxtimes}

\title{Looms}

\author{Ron Aharoni}\thanks{The research of R. Aharoni, E. Berger, and S. Zerbib was supported by BSF grant no.
2006099.\\ The research of R. Aharoni was
supported by  BSF grant no.
2006099, by an ISF grant and by the Discount Bank
Chair at the Technion.\\
The research of S. Zerbib was supported by NSF CAREER award no. 2336239,
NSF award no. DMS-195392,
and Simons Foundation award no. MP-TSM-00002629.
}
\address{Ron Aharoni: Faculty of Mathematics\\ Technion - Israel Institute of Technology, Haifa, Israel. \email[]{raharoni@gmail.com}}

\author{Eli Berger}
\address{Eli Berger: Department of Mathematics\\ University of Haifa, Haifa, Israel.
\email[]{berger.haifa@gmail.com}}

\author{Joseph Briggs}
\address{Joseph Briggs: Department of Mathematics and Statistics\\ Auburn University, Auburn AL, USA.
\email[]{jgb0059@auburn.edu}}

\author{He Guo}
\address{He Guo: Faculty of Mathematics\\ Technion - Israel Institute of Technology, Haifa, Israel.
\email[]{hguo@campus.technion.ac.il}}
\author{Shira Zerbib }
\address{Shira Zerbib: Department of Mathematics \\ Iowa State University, Ames IA,  USA.
\email[]{zerbib@iastate.edu}}
\date{September 6, 2023; revised on July 12, 2024}

\begin{document}

\begin{abstract}
  A pair $(A,B)$ of hypergraphs is called \emph{orthogonal} if $|a \cap b|=1$  for every pair of edges $a \in A,~b \in B$. An orthogonal pair of hypergraphs is called a \emph{loom} if each of its two members is the set of minimum covers of the other. Looms appear naturally in the context of a conjecture of Gy\'arf\'as and Lehel on the covering number of cross-intersecting hypergraphs. We study their properties and ways of construction, and prove special cases of a conjecture that if true would imply the Gy\'arf\'as--Lehel conjecture. 
\end{abstract}

\maketitle

\section{A conjecture on cross-intersecting hypergraphs, and the definition of ``looms''}

A \emph{hypergraph} $H$ is a collection of subsets of a set $V=V(H)$. The elements of $V$ are {\em vertices} and the sets in $H$ are \emph{edges}. We assume  that $V(H)=\cup_{e\in H}e$.
If all edges of $H$ are of size $r$ we say that $H$ is $r$-\emph{uniform}. If $H$ is $r$-uniform and $V(H)$ is the disjoint union of $r$ sets $V_1, \ldots ,V_r$ such that $|e\cap V_i|=1$ for every $1\le i\le r$ and every edge $e$ of $H$, then $H$ is called $r$-\emph{partite}. The sets $V_i$ are then called \emph{sides} of $H$.

A \emph{matching} is a set of disjoint edges. The collection of all matchings in a hypergraph~$H$ is denoted by $\mathcal{M}(H)$.  A matching $M\in \mathcal{M}(H)$ is \emph{perfect}  $\bigcup M=\cup_{e\in M}e=V(H)$.
The matching number $\nu(H)$ of a hypergraph $H$ is $\max_{M\in \mathcal{M}(H)}|M|$. A \emph{cover} of $H$ is a set of vertices meeting all edges of~$H$. The \emph{covering number} $\tau(H)$ of $H$ is the minimal size of a cover of~$H$.

The union of all edges in a maximal matching is a cover, and hence in an $r$-uniform hypergraph~$H$, $\tau(H) \le r\nu(H)$. 
A famous conjecture of Ryser is a sharpening  when $H$ is $r$-partite:
\begin{conjecture}
    If 
$H$ is $r$-partite then $\tau(H) \le (r-1)\nu(H)$.
\end{conjecture}
This conjecture, usually  attributed to Ryser, first appeared in a thesis of one of his students, Henderson~\cite{Henderson}.

The case $r=2$ is the Frobenius-K\"onig-Hall theorem. The case $r=3$ was proved in \cite{ryser3}. The conjecture is open for $r\ge 4$. When~$\nu(H)=1$, Gy\'arf\'as proved the case $r=4$, and Tuza~\cite{tuzar=5} proved the case $r=5$.

A pair of hypergraphs~$(A, B)$ is \emph{cross-intersecting} (a notion coined by Bollob\'as) if $e\cap f\neq\emptyset$ for any $e\in A$ and $f\in B$.

A pair $(a,b)$ of sets is said to be \emph{orthogonal} if $|a \cap b|=1$. We write then $a \perp b$. A pair $(A,B)$ of hypergraphs is \emph{orthogonal} if $a\perp b$ whenever $a \in A, b \in B$. We write then $A \perp B$.
Let $H^\perp = 
    \{e  \subseteq V(H) \mid \{e\} \perp H\}$.
An $r$-uniform  hypergraph $H$ is $r$-partite if and only if $\nu(H^\perp ) = r$.

If $(A,B)$ are cross-intersecting, then
 for any pair of edges $e \in A$ and $f \in B$, $e \cup f$ is a cover of $A \cup B$, and hence if $A$ is $a$-uniform and $B$ is $b$-uniform, then $\tau(A \cup B) \le a+b-1$. 
Professedly  motivated by Ryser's conjecture, Gy\'arf\'as and Lehel conjectured that this bound can be improved when $A$ and $B$ are both $r$-partite,  sharing the same $r$-partition. 

\begin{conjecture}\label{conj:GL}\cite{gyarfasold, lehel, gyarfaschinese}
    \; If $A, B$ are non-empty  cross-intersecting $r$-partite hypergraphs, sharing the same $r$-partition, then $\tau(A \cup B) \le 2r-2$. 
\end{conjecture}
If true, then the conjecture is tight.
An example showing this (possibly essentially the only one) is obtained by taking the vertex set to be the $r \times (2r-2)$ grid, letting $A$ be the set of first $r-1$ columns, and $B$
the set of $r$-tuples meeting all edges of $A$
and all rows of the grid.




In fact,  Conjecture \ref{conj:GL} belongs to a different realm from that of Ryser's conjecture: it is about rainbow matchings.  
 For a system $\ch=(H_1,  \ldots ,H_m)$ of (not necessarily distinct) hypergraphs, let $\nu_R(\ch)$ denote the maximal size of a rainbow matching, namely a matching consisting of a choice of edges from distinct $H_i$s (not necessarily all).  A pair $\ch=(H_1,H_2)$ of non-empty hypergraphs is cross-intersecting if and only if $\nu_R(\ch)<2$.

Haxell~\cite{haxell} proved: 
\begin{theorem}\label{th:penny}
    If $H_1,\dots, H_m$ are $r$-uniform and
$\tau(\bigcup_{i \in I}H_i)>(2r-1)(|I|-1)$ 
for every  $I\subseteq [m]$ then $\nu_R(\ch)=m$ (namely there exists a full rainbow matching).  
\end{theorem}
This explains the interest in $\tau(A \cup B)$ in Conjecture \ref{conj:GL}. 

 Conjecture \ref{conj:GL} can be generalized along the same lines: 

\begin{conjecture}
    \label{conj:genGL}
    If $\ch=(H_1, \ldots, H_m)$ is a family of $r$-partite hypergraphs sharing the same $r$-partition and
$\tau(\bigcup_{i \in I}H_i)>(2r-2)(|I|-1)$ 
for every  $I\subseteq [m]$, then $\nu_R(\ch)=m$.  
\end{conjecture}
Conjecture \ref{conj:GL} is the case $m=2$.
Later we shall meet indications that the condition of $r$-partiteness is unnecessarily strong
--- it may suffice to assume the existence of one set orthogonal to all  $H_i$ (in an $r$-partite hypergraph every side satisfies this).

For a hypergraph~$H$ let $C(H)$ be the set of covers of $H$, and let $C_{min}(H)$ be the set of minimal covers of $H$. For an integer $k$, let 
\[C_k(H):=\{K \in C(H): |K|=k\}.\]
For our later applications, we may assume that $V(C_k(H))=V(H)$ for every~$k$.
Clearly, if $H$ is $r$-uniform, then for every $s$ 
\begin{equation}\label{c2}
    C_r(C_s(H)) \supseteq H.
\end{equation}
For two hypergraphs $K$ and $J$ on the same vertex set, if $K\subseteq J$, then $C_r(J) \subseteq C_r(K)$, and hence  $C_s(C_r(C_s(H)))\subseteq C_s(H)$. Applying \eqref{c2} with $C_r(H)$ replacing $H$, we get $ C_r(H) \subseteq C_r(C_s(C_r(H)))$, and interchanging $r$ and $s$ (it is possible since the above argument works for any number~$r$ and~$s$), we get 
\begin{equation}\label{c3}
   C_s(C_r(C_s(H)))= C_s(H).
\end{equation}

Towards a proof of Conjecture \ref{conj:GL}, we may assume that $A$ and $B$ are orthogonal, 
because if $|e\cap f|\ge 2$ for some $e\in A$ and $f\in B$, then $e\cup f$  forms a cover of $A\cup B$ of size at most $2r-2$. 
Note that $\tau(A),\tau(B)\le r$ as any edge of $A$ is a cover of $B$ and vice versa. 
We may also assume that $\tau(A)=\tau(B)=r$, otherwise a cover of (say) $A$ of size at most $r-1$, together with an edge of~$A$, forms a cover of $A\cup B$ of size at most $2r-2$.
The negation assumption on the conjecture, namely that $\tau(A \cup B)>2r-2$, is preserved under replacing $A$ and $B$ by hypergraphs containing them, and hence we may assume that $A=C_r(B)$ and $B=C_r(A)$. Indeed, let $A'=C_r(B)$, we have $A\subseteq A'$. Let $B'=C_r(A')$. By~\eqref{c2} we have $B\subseteq C_r(C_r(B))=B'$. Furthermore by~\eqref{c3}, we have $A'=C_r(B)=C_r(C_r(C_r(B)))=C_r(C_r(A'))=C_r(B')$. 

We tag a pair of $r$-uniform hypergraphs satisfying all these conditions an $(r,r)$-{\em loom}. This can be generalized to pairs of hypergraphs of not necessarily equal uniformities. 

\begin{definition}
Let $r,s \geq 1$. An $(r,s)$-{\em loom} is a pair $\L=(A,B)$ of orthogonal hypergraphs satisfying: 
\begin{enumerate}
    \item $A$ is $r$-uniform,
    \item $B$ is $s$-uniform,
    \item $\tau(A) = s, \tau(B)=r$, and
    \item $A=C_r(B), B=C_s(A)$. 
\end{enumerate}
\end{definition}

\begin{examples}\label{ex:looms}\hfill
    \begin{enumerate}
        \item The simplest loom 
        is the $(1,1)$-loom $\U$ in which $A=B=\{\{v\}\}$.
        \item \label{eq:anr1loom} For any $r$ let $\V_r=(A,B)$ be the (unique) $(r,1)$-loom, in which 
 $A=\{\{v_1, \ldots .v_r\}\}$ and $B=\{\{v_1\},\ldots \{v_r\} \}$.         
\item\label{eq:anrsloom} More generally, let $A$ be an $r$-uniform matching with $s$ edges, and $B=C_s(A)$ the set of all transversals of $A$. Then $(A,B)$ is an $(r,s)$-loom.
\item\label{eq:a33loom} On the $r\times r$ grid let $A$ consist of the $r$ rows and  $r$ columns, and let $B$ be the set of $r!$ permutation subgrids.
Then $(A,B)$ is an $(r,r)$-loom, a fact provable by induction on $r$. For future reference, we name it $\L_{r,r}$.
Another representation of this loom is  setting $V=E(K_{r,r})$, $A$  the set of stars of size $r$ in $K_{r,r}$,  and $B$  the set of perfect matchings. 
    \end{enumerate}
\end{examples}

A few observations:

 \begin{lemma}
If  $\L=(A,B)$ is an $(r,s)$-loom then $V(A)=V(B)$. 
\end{lemma}
\begin{proof}
    If, say, $v\in V(A)\setminus V(B)$, then $v\in a$ for some $a \in A$ and $a\setminus\{v\}$ is a cover of $B$, contradicting the assumption $\tau(B)=r$. 
\end{proof}

\begin{lemma}\label{lemma:nu2}
If  $\L=(A,B)$ is an $(r,r)$-loom and $r>1$
then $\nu(A)>1$. Moreover, every edge of $A$ has an $A$-edge disjoint from it. 
\end{lemma}

\begin{proof}
    If $e \in A$ does not have an $A$-edge disjoint from it then it belongs to $C_r(A)$, and hence to $B$. This  contradicts the orthogonality condition. 
   \end{proof}

For a hypergraph $H$ and a vertex $v$ of $H$, let 
$star_H(v) =\{e\in H: v\in e\}$.

\begin{lemma}\label{lemma:star=star}
 Let $\L=(A,B)$ be an $(r,r)$-loom. If $star_A(x)\subseteq star_A(y)$, then $star_A(x)=star_A(y)$. 
\end{lemma}
\begin{proof}
    Negation means that there is an edge $e\in A$ containing both $x$ and $y$ and an edge $e'\in A$ containing $y$ but not $x$. Let $f$ be an edge of $B$ containing $x$. Since $f$ meets $e'$ and is orthogonal to $e$, it contains a vertex in $e'\setminus e$. Then $(f\setminus\{x\})\cup\{y\}$ is also a cover of $A$ as all the edges containing $x$ must contain $y$. Thus $f'=(f\setminus\{x\})\cup\{y\} \in B$. Then  $|f' \cap e'|>1$, contrary to the  assumption that $A\perp B$. 
\end{proof}

\begin{lemma}\label{lemma:loompm}
    Let $\L=(A,B)$ be an  $(r,s)$-loom. A matching $M \in \M(A)$ is perfect if and only if $|M|=s$.
\end{lemma}
\begin{proof}
   If $M$ is perfect then its edges meet all vertices of any  $f \in B$. The orthogonality condition implies then that $|M|=|f|=s$. 
For the other direction, assume  $|M|=s$. If there exists a vertex $v\in V(\L) \setminus V(M)$ then any  $f\in B$ containing $v$ is not large enough to meet all edges in $M$.  
\end{proof}

\begin{notation}
    For a function $f: S \to \mathbb{R}_{\ge 0}$ let $|f|=\sum_{s\in S}f(s)$.
\end{notation}

Given a hypergraph $H$, a \emph{fractional matching} of $H$ is a function $f:E(H)\rightarrow \mathbb{R}_{\ge 0}$ satisfying $\sum_{e\in H:v\in e}f(e)\le 1$ for every $v\in V(H)$. It is called {\em perfect} if equality holds for all $v\in V$.
 The \emph{fractional matching number} $\nu^*(H)$ is the maximum of $|f|$  over all fractional matchings $f$ of $H$. The characteristic function of a matching is a fractional matching, hence 
$\nu(H)\le\nu^*(H)$.
A \emph{fractional cover} of $H$ is a function $g:V(H)\rightarrow \mathbb{R}_{\ge 0}$ satisfying $\sum_{v:v\in e}g(v)\ge 1$ for every $e\in H$. The \emph{fractional covering number} $\tau^*(H)$ is the minimum of $|g|$ over all fractional covers ~$g$ of $H$. 
The characteristic function of a cover is a fractional covering, hence 
$\tau^*(H)\le\tau(H)$.
By LP duality, $\tau^*(H)=\nu^*(H)$, 
so
\[   \nu(H)\le\nu^*(H) = \tau^*(H)\le \tau(H).   \]

\begin{lemma}\label{lemma:loompfm}
       In an $(r,s)$-loom $\L=(A,B)$, $A$ has a perfect fractional matching if and only if $\nu^*(A)=s$.
\end{lemma}
\begin{proof}
Given  $b \in B$, for any fractional matching $f$ of $A$, 
\[ |f| = \sum_{u\in b}\sum_{a\in A:u\in a}f(a).\]
The right-hand side is $s$ if and only if  $f$ is saturated at each $u\in b$. The lemma follows from this and the fact that $\bigcup B =V(\L)$.
\end{proof}

Of course, symmetric claims of these lemmas hold for $B$.

Looms are relevant to  \refCon{conj:GL}, but they are  also interesting  on their own merit. Presently we know precious little about them.
A central question we do not know the answer to is whether each component of a loom must have a perfect fractional matching. As we shall see in~\refS{sec:fmoflooms}, a positive answer would imply 
\refCon{conj:GL}.

\section{Bounding the fractional matching number}
Let us indeed start with the fractional case of \refCon{conj:GL}.  It turns out to be true, with a margin: it does not require $r$-partiteness, and the $2$-factor is redundant.

\begin{theorem}\label{mainineq}
   If $A$ and $B$ are cross-intersecting $r$-uniform hypergraphs then $\tau^*(A \cup B) \le r$.
\end{theorem}
 This was proved in \cite{ak}, and generalized in \cite{abm}
to:

\begin{theorem}\label{th:abm}
  Let $\ch=(H_1,H_2, \ldots ,H_m)$ be a system of $r$-uniform hypergraphs.  If 
$\tau^*(\bigcup_{i \in I}H_i)>r(|I|-1)$ 
for every  $I\subseteq [m]$ then $\nu_R(\ch)=m$ (namely there exists a full rainbow matching).  
\end{theorem}
A standard deficiency argument yields a generalization:

 \begin{theorem}
If $\nu_R(\cH)<m$ then  there exists a subset $I$ of $[m]$ for which
$\tau^*(\cup_{i\in I}H_i)\le r(|I|-(m-\nu_R(\cH)))$.
\end{theorem}

 If $\nu_R(\cH)=1$
 and $H_i\neq \emptyset$ for all $i$, then necessarily the set $I$ in the theorem is  $[m]$. This  yields:

\begin{theorem}\label{thm:taustar}\cite{ak,abm} 
If $\nu_R(\cH)=1$
 and $H_i\neq \emptyset$ for all $i$, 
 then $\tau^*(\bigcup_{1\le i \le m}H_i)\le r$.
\end{theorem}

In \cite{ak} it was shown that the fractional cover witnessing this result can be chosen as a convex combination of $\{\chi_e\}_{e \in H_i}$ for some $i$. ($\chi_e$ is the characteristic function of $e$, namely $\chi_e(v)=1$ if $v\in e$  and $\chi_e(v)=0$ if $v \notin e$.)

\begin{corollary}\label{cor:taustarlers}
    If $A, B$ are cross-intersecting, $A$ is  $r$-uniform 
    and $B$ is  $s$-uniform 
   then $\tau^*(A \cup B) \le \max(r,s)$.
\end{corollary}

\begin{proof}
    Assume $r \ge s$. Replace every edge  $b \in B$ by $b \cup dummy(b)$, where 
$dummy(b)$ is a set of size $r-s$, disjoint from $V(A) \cup V(B)$ and $dummy(b)
\cap dummy(c)=\emptyset$ whenever $b \neq c$. Apply now Theorem \ref{mainineq}.    
\end{proof}

It can be shown that equality holds in \refT{thm:taustar} if and only if $\nu^*(H_i)=r$ for some $1\le i \le m$. Let us state and prove this just for $m=2$.  

\begin{theorem}\label{thm:taustarmax}
If $H_1, H_2$ are cross-intersecting $r$-uniform hypergraphs and $\tau^*(H_1 \cup H_2)=r$ then    $\max_{i=1,2}\nu^*(H_i)=r$.
\end{theorem}

\begin{proof}
Assume, for contradiction, that

\begin{equation}\label{negation}
\max_{i=1,2}\nu^*(H_i)<r.
\end{equation}

Let $V=V(H_1\cup H_2)$. Let $C_i=conv(\{\chi_e\in \mathbb{R}^{|V|} \mid e \in H_i\})$ be the convex hull of the characteristic functions for $i=1,2$. By the cross-intersection assumption $\chi_{e_1}\cdot \chi_{e_2}\ge 1$ for any $e_1\in H_1$ and $e_2\in H_2$. Hence $w_1\cdot w_2\ge 1$ for any $w_1\in C_1$ and $w_2\in C_2$. 

Let $u_i$ be the shortest vector in $C_i$ ~($i=1,2$).
Since $H_i$ is $r$-uniform, we have $u_i \cdot \b1=r$,  where $\b1\in\mathbb{R}^{|V|}$ is the all $1$s vector. Since
$u_1 \cdot u_2 \ge 1$,  at least one of $u_1,~u_2$ has length at least $1$.

Assume first that one of $u_1, u_2$ is strictly longer than $1$, say $||u_1||>1$.
Since $u_1$ is the shortest vector in $C_1$, 
we have $u_1\cdot \chi_e \ge u_1\cdot u_1>1$ for every $e$ in $H_1$. 
Therefore, for some $0<\alpha <1$ we have $\alpha u_1\cdot \chi_e \ge 1$ for every $e\in H_1$. By \eqref{negation}, there exists a fractional cover $g$ of $H_2$ satisfying $|g|<r$. Then   $w=\alpha u_1+ (1-\alpha)g$ is a fractional cover of $H_1\cup H_2$. Since  $|w| = w\cdot \b1 <\alpha r+(1-\alpha)r =r$,  this contradicts the assumption of the theorem.

We may thus assume that $||u_1||=||u_2||=1$.
This implies $u_1=u_2=:u$, since otherwise $u_1 \cdot u_2<1$. The vector $u$ is then  a fractional cover of $H_1\cup H_2$  of weight $u_1\cdot\b1=r$.

\begin{claim}
Let $supp(u):=\{v\in V\mid u(v)>0\}$. Then
$|supp(u)| =r^2$, and $u(v)=\frac{1}{r}$ for every $v \in supp(u)$.
\end{claim}
\begin{proof}[Proof of the claim]
We first show that $|supp(u)| \le r^2$.
Let $f$ be a fractional matching of  $H_1\cup H_2$ with $|f|=r$. By complementary slackness, $\sum_{e:v \in e}f(e)=1$ for every  $v \in supp(u)$. Therefore
\begin{align*}
    |supp(u)|= \sum_{v \in supp(u)}1 &=\sum_{v \in supp(u)} \sum_{v \in e} f(e)\\
    &\le \sum_{v\in V}\sum_{e:v\in e}f(e)=\sum_{e}\sum_{v:v\in e}f(e)=r\sum_{e}f(e)=r^2.
\end{align*}

For the inverse inequality, by the Cauchy-Schwarz inequality we have:
$$r=u \cdot \chi_{supp(u)} \le ||u|| \times ||\chi_{supp(u)}||= \sqrt{|supp(u)|}$$
 so $|supp(u)| \ge r^2$. The claim that $u$ is constant follows from the characterization of equality in the Cauchy-Schwarz inequality.
\end{proof}
To deduce \refT{thm:taustarmax} from the claim, write $u=\sum_{e \in H_1} \alpha_e \chi_e$ with $\sum_{e \in H_1}\alpha_e=1$, and let $f: E(H_1) \to \mathbb{R}_{\ge 0}$ be defined by $f(e)=r\alpha_e$. Since  $\sum_{e\in H_1:v\in e}f(e)=ru(v)=1$ for every $v\in V(H_1)$, ~~$f$ is a fractional matching of $H_1$. Its  weight $f$ is $\sum_{e\in H_1}f(e)=r\sum_{e\in H_1}\alpha_e=r$, which proves the theorem.
\end{proof}

\begin{remark}
The condition  $\tau^*(H_1 \cup H_2)=r$ is necessary: it is not always true that if $A,B$ are cross-intersecting then  $\nu^*(A \cup B)=\max(\nu^*(A), \nu^*(B))$. For example, if $H_1=\{ab,bc\}$ and $H_2=\{ac\}$ then $\nu^*(H_i)=1~~(i=1,2)$ while $\nu^*(H_1\cup H_2)=\frac{3}{2}$. 
\end{remark}

\section{Many mutually cross-intersecting  hypergraphs}

Famously, there can be no more than $r+1$  mutually orthogonal $r$-uniform matchings of size $r$ each. This is usually expressed in terms of mutually orthogonal Latin squares --- there can be no more than $r-1$ of those of order $r$ (see, e.g.,~\cite[Theorem 6.29]{mols}). 
The bound is attained when there exists an $(r+1)$-uniform projective plane, in particular when $r$ is a power of a prime. 
Since a matching of size $r$ has $\tau^*=r$, the following  is a fractional generalization of this theorem.

   \begin{theorem}\label{thm:smalltau*}
  Let  $(H_1, \ldots, H_m)$ be a family of pairwise cross-intersecting $r$-uniform hypergraphs, where $r\ge 2$. If $m>r+1$ then $\tau^*(H_i)< r$ for some $1\le i \le m$.
   \end{theorem}

\begin{proof}
 Since  $e\cap V(H_i)$ is a cover of $H_i$ for every $e \in H_j$, there holds $|e\cap V(H_i)|=r$,  so  $e\subseteq V(H_i)$. This proves that   $V(H_i)=V(H_j)$ for all $i,j \le m$.
 
 Every $H_i$ is covered by any edge from any $H_j, ~ j \neq i$, hence  $\tau^*(H_i)\le \tau(H_i)\le r$. Thus, the negation assumption on the theorem is that 
    $\tau^*(H_i)=r$ for all $i \le m$.  
Let $m_i: H_i\to \mathbb{R}$ be a fractional matchings of $H_i$ satisfying $|m_i|=r$.
Removing edges of zero weight we may assume that $H_i=supp(m_i)$. 
Let $w_i(v)=\sum_{e\in H_i:v\in e} \frac{m_i(e)}{r}$. 

\begin{claim}
    For any $j \le m$, $w_i$ is a minimum fractional cover of $H_j$.
\end{claim}

\begin{proof}[Proof of the claim]
    
For any $f\in H_j$ we have
\begin{equation}\label{eq:coveratavertex}
    \sum_{v\in f} w_i(v) =\sum_{v\in f} \sum_{e\in H_i:v\in e} \frac{m_i(e)}{r} 
=\sum_{e\in H_i}\frac{m_i(e)}{r}|f\cap e| \geq \sum_{e\in H_i} \frac{m_i(e)}{r} = 1.
\end{equation}
On the other hand \[|w_i| = \sum_{v\in V}\sum_{e\in H_i}\frac{m_i(e)}{r}=\sum_{e\in H_i}\sum_{v\in e}\frac{m_i(e)}{r}= \sum_{e\in H_i} \frac{m_i(e)}{r}\cdot r = r,\]

proving the claim. \end{proof}

By complementary slackness, for each $h\in supp (m_j)=H_j$, we have $\sum_{v\in h} w_i(v)=1$. Thus in~\eqref{eq:coveratavertex} we have equality, and hence 
\begin{equation}\label{eq:intersectionis1}
    |h\cap e|=1
\end{equation}
for any $e\in H_i$ and $h\in H_j$.

Let $v\in V(H_1)$ and $e_1\in H_1$ be such that $v\not\in e_1$ (we may assume that there exist such, since otherwise $|H_1|=1$ and $\tau^*(H_1)=1<2$.) For each $2\le j\le r+1$ there exists $e_j\in H_j$ such that $v\in e_j$. But $H_1$ and $H_j$ are cross-intersecting, therefore $e_j\cap e_1\neq\emptyset$. By pigeonhole, there exist $x\in e_1$ and $2\le i< j\le m$ such that $x\in e_i\cap e_1$ and $x\in e_j\cap e_1$. Therefore $\{v,x\}\subseteq e_i\cap e_j$,  contradicting ~\eqref{eq:intersectionis1}.
\end{proof}

\begin{question}
    We know that the result is sharp for $r$ a prime power. Is it sharp  also for other values of $r$? Namely, are there examples with $m=r+1$ and $\nu^*(H_i)=r$ ($1\le i \le m$) for general $r$? 
\end{question}

\begin{definition}
    For integers $r,m$
    let $g(r,m)$   be the maximum, over all $m$-tuples $(H_1, \ldots ,H_m)$ of pairwise cross-intersecting $r$-uniform hypergraphs, of $\min_{1\le i\le m}\tau^*(H_i)$.
\end{definition}
 For $r$ having a projective plane $\P_r$ of uniformity $r$, we have $g(r,m) \ge r-1+\frac{1}{r}$, by taking $H_i=\P_r$ for all $i$.
 
 \begin{conjecture}
     For $m \ge r+2$ we have $g(r,m)\le r-1+\frac{1}{r}$, with equality if and only if there is a projective plane of uniformity $r$. 
 \end{conjecture}

 The conjecture is true for $r=2$, since in this case $\tau^*(H_i)<2$ means $\nu(H_i)=1$, implying $\tau^*(H_i)\le \frac{3}{2}$.

 \begin{theorem}
         For every $r$ there exists $m=m(r)$ such $g(r,m) \le r-1+\frac{1}{r}$. 
         \end{theorem}
\begin{proof}
     Let $C=C(r)$ be such that in any $r$-uniform hypergraph of size $C$ there exists a $\Delta$-system with at least $r+1$ leaves. (A $\Delta$-system with $\ell$ leaves is a hypergraph~$D$ consisting of $\ell$ edges for which there exists a ``core'' set $K$, such that $e \cap f=K$
     for all pairs $e\neq f \in D$. The existence of $C$ as above was proved in \cite{delta_erdos}). 
    If there exists~$H_i$ of size larger than~$C$, then $H_i$ contains a $\Delta$-system with $r+1$ leaves, and then its core is a cover for every other $H_j$, and hence $\tau(H_j)\le r-1$ for all $j \neq i$.
    
    Thus we may assume that $|E(H_i)|\le C$ for all $1\le i \le m$. For large enough $m$ there exist then $j_1\neq j_2$ for which  $H_{j_1}=H_{j_2}$. Then  $\nu(H_{j_1})=\nu(H_{j_2})=1$, and hence by a theorem of F\"uredi \cite[Theorem]{furedi}, we have $\tau^*(H_{j_1})\le r-1 + \frac{1}{r}$. 
    \end{proof}

\section{The fractional matching number of looms}\label{sec:fmoflooms}
Apart from the present section, the rest of the paper is devoted to the construction of looms. But before embarking on this project, we present a conjecture that implies \refCon{conj:GL}. This way we can accompany each construction with a verification of the conjecture in the constructed loom.

For a loom $\L=(A,B)$ let 
\[\tau(\L):=\tau(A \cup B),\quad \nu(\L):=\nu(A \cup B) \quad\text{and}\quad \tau^*(\L):=\tau^*(A \cup B).\] Of course, $\nu(\L)=\max(\nu(A), \nu(B))$. 

By \refC{cor:taustarlers} if $\L=(A,B)$ is an $(r,s)$-loom then 
$\tau^*(\L)\le \max(r,s)$.
It may well be that equality holds.  

\begin{conjecture}\label{conjmain} 
If $\L=(A,B)$ is an $(r,s)$-loom then $\tau^*(\L)=\max(r,s)$.  
\end{conjecture}

By \refT{thm:taustarmax}, if $r=s$ then this would imply that 
 $\max(\tau^*(A), \tau^*(B))=r$.  More generally:

\begin{conjecture}\label{conjmain1} 
If $\L=(A,B)$ is an $(r,s)$-loom then $\tau^*(A)=s, \tau^*(B)=r$.  
\end{conjecture}
\begin{proposition}
 If  \refCon{conjmain1} is true for an $(r,s)$-loom $(A,B)$, then $|V(A)|=|V(B)|=rs$. 
\end{proposition}
\begin{proof}
Let $w: A\rightarrow \R_{\ge 0}$ be a  fractional matching of size $s$. By~\refL{lemma:loompfm},  $\sum_{e\in A:v\in e}w(e)=1$ for every $v\in V$. Thus we have 
 \[|V(A)| = \sum_{v\in V(A)} 1 = \sum_{v\in V(A)} \sum_{e\in A:v\in e} w(e) = \sum_{e\in A}\sum_{v\in V: v\in e}w(e)= \sum_{e \in A} r w(e) = rs.\]

\end{proof}

In particular, \refCon{conjmain1}
implies \refCon{conj:GL}, since if the hypergraphs are $r$-partite and $|V(\L)|=r^2$, then there exists a side  of size at most $r$, and since every side is a cover this implies $\tau(\L) \le r$. Note that this result is not stronger than the original conjecture (bounding the covering number by $2r-2$), because we used the negation assumption on the latter.

\subsection{Pinnability}
 
\begin{definition}
   A hypergraph $H$ is said to be {\em pinnable} if $H^\perp \neq \emptyset$.   
\end{definition}

We suspect that in ~\refCon{conj:GL} the  milder assumption of pinnability suffices: 
\begin{conjecture}\label{conj:GLpin}
    If $A, B$ are cross-intersecting $r$-uniform hypergraphs and $A \cup B$ is pinnable, then $\tau(A \cup B) \le 2r-2$. 
\end{conjecture}

This conjecture can be formulated also for hypergraphs of different uniformities. 
\begin{conjecture}
    
\label{genabpin}
 If $A$ is $r$-uniform, $B$ is $s$-uniform, $A$ and $B$ are cross-intersecting, $V(A)=V(B)$ and $A \cup B$ is pinnable, then $\tau(A \cup B) \le r+s-2$.   
\end{conjecture}
\begin{theorem}
 \refCon{genabpin} is true for $s=2$.    
\end{theorem}

\begin{proof}
Let $p\in (A\cup B)^\perp$ be a pinning set for $A \cup B$.
Then $B$ is a bipartite graph with sides $p, V\setminus p$. 
Let $v \in p$. By the assumption that $V(A)=V(B)$, $v$ is contained in a triple  $e \in A$. By the cross-intersection property $e\setminus\{v\} \supseteq N_B(p \setminus \{v\})$, and hence 
$|N_B(p \setminus \{v\})|\le r-1$, so $v+N_B(p \setminus \{v\})$ is a covering set of $A\cup B$ of size at most~$r$.    
\end{proof}

The condition $V(A)=V(B)$ is essential --- the theorem is not true, for example, if one of the hypergraphs is empty. 

 A  strengthening    of \refCon{conj:genGL} is:

\begin{conjecture}\label{pinconj}
    Let $\ch=(H_1, \ldots ,H_m)$ be a family of  $r$-uniform hypergraphs. If $H_1\cup\dots\cup H_m$ is pinnable and $\tau(\cup_{i\in I}H_i) > (2r-2)(|I|-1)$ for all $ I\subseteq [m]$  then $\ch$ has a full rainbow matching, i.e, $\nu_R(\ch)=m$ . 
\end{conjecture}

\begin{theorem}
   For an $(r,s)$-loom $(A,B)$, if $\tau^*(A)=s$ then $B=A^\perp$. 
\end{theorem}

\begin{proof} Let $p \in A^\perp$, we wish to show that $ p \in B$. It suffices to show that $|p| \le s $. 
    Let $f$ be a maximum fractional matching, namely of size $s$. By complementary slackness $|f|=|p|$, hence $|p| \le s$, as desired. 
\end{proof}

\begin{theorem}
If $\L=(A,B)$ is an $(r,r)$-loom and $\tau^*(A)=r$, then $|p|=r$ for every set  $p \in (A \cup B)^\perp$.
\end{theorem}
\begin{proof}
    Let $f$ be a fractional matching of size $r$ of $A$, then by~\refL{lemma:loompfm} it is a perfect fractional matching. Let $p\in (A \cup B)^\perp$. We have
    \[  |p|=\sum_{v\in p}\sum_{e\in A: v\in e}f(e)=|f|=r.  \]
   \end{proof}

\section{Constructions}

Looms are varied, and yet scarce enough to make their construction  challenging. 
The simplest loom  $\U$, in which $A=B=\{\{v\}\}$, has already been mentioned.  From it we shall be able to construct many others, using simple operations, the first of which is {\em composition}. 

\subsection{Composition of looms}
Given two hypergraphs $A$ and $C$ on disjoint vertex sets, their {\em join} $A*C$ is $\{a \cup c \mid a \in A, c \in C\}$. 
Let $\L_1=(A,B_1)$ be an $(a,b)$-loom and $\L_2=(C,B_2)$ be a $(c,b)$-loom, where $V_1=V(A)=V(B_1)$ is disjoint from $V_2=V(C)=V(B_2)$.
Let $S=A * C$ and $T=B_1 \cup B_2$. The pair $\L_1 \cplus_1 \L_2=(S,T)$ is called the {\em $1$-composition} of $\L_1$ and $\L_2$. The $1$ indicates that the $*$ operation is applied to the first coordinate. A similar definition and notation apply to \emph{$2$-compositions}. 

\begin{lemma}
  For an $(a,b)$-loom $\L_1$ and a $(c,b)$-loom $\L_2$ on disjoint vertex sets, $\L_1 \cplus_1 \L_2$ is an  $(a+c,b)$-loom.
\end{lemma}

The proof is easy, and is contained in the proof of ~\refT{thm:blowupisaloom} below.

A loom that is the composition of two looms is said to be {\em decomposable}.
A hypergraph is defined to be \emph{connected} if its $1$-skeleton (the collection of subsets of size $2$ of the edges) is a connected graph. 

For sets $S$ and $T$, we define
\[ S\upharpoonright T=\{s\cap T\mid s\in S\}. \]

\begin{lemma}\label{lemma:decomposablenotconnected}
A loom  is decomposable if and only if one of its components is not connected.
\end{lemma}

\begin{proof}
    If $\L=(A,B)$ is decomposable, then by definition, at least one of $A$ or $B$ is not connected.

   For the other direction, let $\L=(A,B)$ be an $(r,s)$-loom in which one of the components, say $B$, is not connected. That is, $B=B_1\cup B_2$, where $V(B_1) \cap V(B_2)=\emptyset$.
   
   Let $A_1=A\upharpoonright V(B_1)=\{a\cap V(B_1)\mid a\in A \}$ and $A_2=A\upharpoonright V(B_2)=\{a \cap V(B_2)\mid a\in A \}$. We shall prove that $(A_1,B_1), (A_2,B_2)$ are looms and $(A,B)=(A_1,B_1)\cplus_1 (A_2,B_2)$.

\begin{claim}\label{claim:r1r2r}
For $i=1,2$, $A_i$ is $r_i$-uniform for some $r_i$ satisfying $r_1+r_2=r$.
\end{claim}
\begin{proof}[Proof of the claim]
    Let $e, f \in A$, and assume for negation that (say) $|f\cap V(B_1)|<|e\cap V(B_1)|$. Then $(f\cap V(B_1))\cup (e\cap V(B_2))$ is a cover of $B$ of size smaller than~$r=|e|$, contrary to the assumption that $\tau(B)=r$.
\end{proof} 
Let $r_i$  be the uniformity of $A_i$ for $i=1,2$, which by~\refCl{claim:r1r2r} satisfies $r_1+r_2=r$.
\begin{claim}
$(A_i,B_i)$ is an $(r_i,s)$-loom for $i=1,2$.    
\end{claim}    
\begin{proof}[Proof of the claim]

First we prove that $(A_i, B_i)$ are orthogonal. Let $a_i \in A_i, b \in B_i\subseteq B$. Then $a_i = a \cap V(B_i)$ for some $a \in A$. Since $(A,B)$ are orthogonal, $|a \cap b|=1$ and $a\cap b\subseteq V(B_i)$, meaning that $a_i\cap b=a\cap b$ and $|a_i \cap b|=1$.

Next we show that $\tau(B_i)=r_i$ and $C_{r_i}(B_i)=A_i$ for $i=1,2$. 
For any $a_i\in A_i$, we have $a_i=a\cap V(B_i)$ for some $a\in A$. Since $a$ is a cover of $B_i$, then  $a_i=a\cap V(B_i)$ is also a cover of $B_i$. Therefore $A_i\subseteq C(B_i)$ and $\tau(B_i)\le r_i$. Suppose $e$ is a cover of $B_i$ of size at most $r_i$. We take $a\in A$ and $a_j=a\cap V(B_j)$ is a cover of $B_j$ for $j\neq i$. Hence $e\cup a_j$ is a cover of $B=B_i\cup B_j$. Therefore $\tau(B_i)=r_i$ (otherwise we have a cover of $B$ of size less than $r_i+r_j=r$) and $e\cup a_j=a'$ for some $a'\in A$ (otherwise $C_r(B)\neq A$) so that $e=a'\cap V(B_i)\in A_i$, which proves $C_{r_i}(B_i)=A_i$. 

Finally we show that $\tau(A_i)=s$ and $C_s(A_i)=B_i$. Since a cover of $A_i$ is a cover of $A$, we have $\tau(A_i)\ge \tau(A)\ge s$ and $C_s(A_i)\subseteq C_s(A)\upharpoonright V(B_i)=B_i$. 
On the other hand, any element of $B_i$ is a cover of $A_i$, therefore $\tau(A_i)=s$ and $B_i\subseteq C_s(A_i)$. Hence $C_s(A_i)=B_i$.
\end{proof}
 
The above proof shows that $A_1*A_2\subseteq C_r(B)=A$, and by definition we have $A\subseteq A_1*A_2$. Therefore $A=A_1*A_2$ and then $(A,B)=(A_1,B_1)\cplus_1 (A_2, B_2)$. 
\end{proof}

On the way we proved:

\begin{lemma}
   Let $(A,B)$ be a loom.  If $C$ is a connected component of $B$, then $(A \upharpoonright V(C),C)$ is a loom.
\end{lemma}

The composition $\U^{\cplus_1 r}$ of $\U$ with itself $r$ times
is  $\V_r$, from \refE{ex:looms}~\eqref{eq:anr1loom}.
The $s$-fold $\cplus_2$-composition of $\V_r$ with itself
results in the loom of \refE{ex:looms}~\eqref{eq:anrsloom}.

\subsection{Blow-ups: a generalization of composition}\label{subsec:genblowup}
Composition  produces  new looms from given ones. There is a more general operation of this sort, which we call ``blow-up''. It takes a cross-intersecting pair $\P$ of hypergraphs and replaces each vertex with a loom. Under certain  conditions, the resulting pair of hypergraphs is a loom. A sufficient condition is that $\P$ itself is a loom and we replace its vertices by looms of the same size, but there are more general settings in which this is attained.

Let $\cP=(A,B)$ be a pair of orthogonal hypergraphs (not necessarily uniform), satisfying $\bigcup A=\bigcup B=V=[n]$. Then $A \subseteq C_{min}(B)$ (and  $B \subseteq C_{min}(A)$): $a-v$ is not a cover for every $v \in a$, since there is an edge of $B$ meeting $a$ just in $v$.

Let $\cP_i=(A_i,B_i)$ with $1\le i\le n$ on disjoint vertex sets, and let
$$C= \bigcup \{ A_{i_1}*\cdots*A_{i_p}\mid\{i_1,\dots, i_p\}\in A \},$$ $$D= \bigcup \{ B_{j_1}*\cdots*B_{j_q}\mid \{j_1,\dots, j_q\}\in B  \}.$$ 
We then call $(C,D)$ a \emph{blow-up} of $\P$ and denote it by $\cP[\cP_1,\dots, \cP_n]$.

\begin{remark}
To see that compositions are a special case,  note the following.  If $\L=(A,B)$ is the (2,1)-loom  $\V_2$,  and the uniformities of the second components of $\L_1$ and $\L_2$ are the same, then $\L[\L_1,\L_2]=\L_1\cplus_1\L_2$.
\end{remark}

Theorem \ref{thm:blowupisaloom} below provides sufficient conditions for blow-ups to be looms. Before delving into them, here is an example 
against which they can be checked.

For $x,y,z<10$  we write below $xy$ for $\{x,y\}$ and $xyz$ for
$\{x,y,z\}$.

\begin{example}\label{ex:vane}

Number the vertices of the $3\times 3$ grid $1,2,\ldots ,9$ so that the columns are $147,258,369$ and the rows $123,456,789$.

Let    
\[ C=\{147,258,369,159,158,247,259,368 \} \]
and 
\[D=\{123,456,789,357,126,345,489,567\}.\]
 Denote  the pair $(C,D)$ by 
$\V_{3,3}$.

To see that $\V_{3,3}$ is a blow-up, let $\cP=(A,B)$ be a pair of hypergraphs on $\{x_1,x_2,x_3,x_4,x_5\}$, where 
\[A=\{x_1x_4,x_2x_5,x_1x_3x_5\}\text{ and } B=\{x_1x_2,x_4x_5,x_2x_3x_4 \}.\]
Let $\L_1=(\{\{1\},\{2\} \},\{12\})$ and $\L_5=(\{\{8\},\{9\} \},\{89\})$ be $(1,2)$-looms, $\L_2=(\{36\},\{ \{3\},\{6\}\})$ and $\L_4=(\{47\},\{\{4\},\{7\} \})$ be $(2,1)$-looms, and $\L_3=(\{\{5\} \},\{\{5\}\})$ be a $(1,1)$-loom.
Then $\cP[\L_1,\L_2,\L_3,\L_4,\L_5]=(C,D)$.

We claim that  $\V_{3,3}$ is a $(3,3)$-loom. This can be checked directly, but it also follows from the next theorem.
\end{example}

\begin{theorem}\label{thm:blowupisaloom}
    Let $\cP=(A,B)$ be an orthogonal pair of hypergraphs satisfying $\bigcup A=\bigcup B$.
    Let $\L_i=(A_i,B_i)$ be $(r_i,s_i)$-looms for $1\le i\le n$. If $\cP[\L_1,\dots,\L_n]=(C,D)$ satisfies  
    \begin{enumerate}
\item \label{eq:CDuniform} $C$ and $D$ are uniform - say $C$ is $c$-uniform and $D$ is $d$-uniform, and 
\item \label{eq:minimalnotinB} For any minimal cover $f$ of $A$ that is not in $B$, we have $\sum_{j\in f}s_j>d$, and 
\item \label{eq:minimalnotinA} For any minimal cover $e$ of $B$ that is not in $A$, we have $\sum_{i\in e}r_i>c$.    
    \end{enumerate}
  Then $(C,D)$ is a $(c,d)$-loom.
\end{theorem}

Note that in Example \ref{ex:vane}, $\{x_1,x_5\}$ is a minimal cover of $A$ and is not in $B$, but the uniformity of the join of the second components of $\L_1$ and $\L_5$ is 4, which is greater than 3. Therefore the blow-up satisfies the assumption of~\refT{thm:blowupisaloom}, which implies that $\V_{3,3}$ is a loom.

\begin{proof}[Proof of Theorem \ref{thm:blowupisaloom}]
    \begin{claim}\label{claim:blowuporth}
    $C \perp D$. 
    \end{claim}
    \begin{proof}[Proof of the claim]
        Let $e\in C$ and $f\in D$. By the  construction $e=e_{i_1}\cup \cdots \cup e_{i_p}$ where $\{i_1,\dots,i_p\}\in A$ and $e_{i_\ell}\in A_{i_\ell}$ for every $\ell$, and $f=f_{j_1}\cup\cdots\cup f_{j_q}$ where $\{j_1,\dots,j_q \}\in B$ and $f_{j_\ell}\in B_{j_\ell}$ for every $\ell$. 
    Then by $A\perp B$, we have $\{i_1,\dots,i_p\}\cap\{j_1,\dots, j_q\}=\{k\}$, and by the disjointness of $V(\L_i)$ we have $e\cap f= e_k\cap f_k$. Since $(A_k,B_k)$ is a loom,  $|e\cap f|=|e_k\cap f_k|=1$, as claimed.
    \end{proof}

 It remains to show that $\tau(C)=d$, and $C_d(C) \subseteq D$ (so that Claim \ref{claim:blowuporth} implies $C_d(C) =D$). $\tau(D)=c$ and $C_c(D) \subseteq C$ will follow symmetrically.

\begin{claim}\label{claim:geneachintersection}
    If $f$ is a minimum cover of $C$, then for each $j\in [n]$, if $f\cap V(B_j)$ is non-empty then it belongs to $B_j$.
\end{claim}
\begin{proof}[Proof of the claim]
Suppose $f\cap V(B_j)\neq\emptyset$. We claim that it is a minimum cover of $A_j$. 

First if it is not a cover of $A_j$,
let $f'=f\setminus V(B_j)$. By the minimality of $|f|$, there is an edge $\cup_{i\in a}e_i\in C$, where $a\in A$ and $e_i\in A_i$, not covered by $f'$.
In particular, we have $j\in a$. Then $e=\cup_{i\in a\setminus\{j\}}e_i\cup e_j'\in C$ is not covered by $f$, where $e_j'\in A_j$ is an edge not covered by $f\cap V(B_j)$, a contradiction to the fact that $f$ is a cover of $C$.

 Now if $f\cap V(B_j)$ 
 is not a minimum cover of $A_j$, we can take a minimum cover $e_j$ of $A_j$ and set $f'=\Big(f\setminus \big(f\cap V(B_j)\big)\Big)\cup e_j$. Note that the size of $f'$ is strictly smaller than $f$ but $f'$ covers the same edges of $C$ as $f$, contradicting the minimum assumption on $f$. 
\end{proof}

\begin{claim}\label{claim:genwholeintersection}
    If $f$ is a minimum cover of $C$, then $J=\{j:f\cap V(B_j)\neq \emptyset\}$ is a minimal cover of $A$ and is in $B$.
\end{claim}
\begin{proof}[Proof of the claim]
    If $J$ is not a cover of~$A$, then there exists $a\in A$ such that $a\cap J=\emptyset$. Choose an edge $e_i\in A_i$ for each $i\in a$. Then 
    $f\cap (\cup_{i\in a} e_i)=\emptyset$, a contradiction to the assumption that $f$ is a cover of~$C$.
      So $J$ is a cover of $A$. If it is not minimal, we can remove a subset $R$ so that $J\setminus R$ is a minimal cover of $A$, then by~\refCl{claim:geneachintersection}, $f\setminus (\cup_{j\in R}f\cap V(B_j))$ is still a cover of $C$ but is of smaller size than $f$, a contradiction to the minimum assumption of~$f$.
      To see that $J$ is in $B$, negating this statement means that  $f$ is of size strictly larger than $f'$ for any $f'\in D$. But $f'$ is a cover of $C$, a contradiction to the fact that $f$ is a minimum cover.
\end{proof}

Combining~\refCl{claim:geneachintersection} and~\refCl{claim:genwholeintersection} with the definition of the blow-up $(C,D)=\cP[\L_1,\dots,\L_n]$, we have $\tau(C)=d, C_d(C)\subseteq D$ so that $C_d(C)=D$ by~\refCl{claim:blowuporth}, and symmetrically we have $\tau(D)=c, C_c(D)=C$ --- the remaining conditions required on $(C,D)$ for being a loom. This completes the proof of the theorem.    
 \end{proof}

\begin{corollary}
       Let $\L=(A,B)$ be a $(p,q)$-loom on the vertex set $V=[n]$, and for each $1\le i\le n$ let $\L_i=(A_i,B_i)$ be an $(r_i,s_i)$-loom, where $V(\L_i)$ are disjoint. If $r_{i_1}+\cdots +r_{i_p}= c$ for every $\{i_1,\dots,i_p\}\in A$ and $s_{j_1}+\cdots + s_{j_q}=d$ for every $\{j_1,\dots, j_q\}\in B$, and $r_{i_1}+\cdots +r_{i_{p+1}}>c $ for any distinct vertices $i_1,\dots,i_{p+1}\in V$ and $s_{j_1}+\dots+s_{j_{q+1}}>d$ for any distinct vertices $j_1,\dots, j_{q+1}\in V$, then $(C,D)=\L[\L_1,\dots,\L_n]$ is a $(c,d)$-loom.
\end{corollary}
\begin{proof}
\eqref{eq:CDuniform} is valid by the assumption of the corollary. By the symmetry between~\eqref{eq:minimalnotinB} and~\eqref{eq:minimalnotinA},
we only need to verify~\eqref{eq:minimalnotinA} of~\refT{thm:blowupisaloom}:
if $e$ is a minimal cover of $B$ but is not in $A$, then $e$ is not a minimum cover of $B$, which means $|e|>p$. Therefore by the assumption of the corollary, $\sum_{i\in e}r_i > c$. Hence $\L[\L_1,\dots,\L_n]$ satisfies~\eqref{eq:minimalnotinA} of~\refT{thm:blowupisaloom}.
\end{proof}

Next we show that having perfect matchings and fractional perfect matchings are preserved (in a sense to be stated below) by blow-ups.

\begin{theorem}\label{thm:pmiffpmineach}
Let $\L=(A,B)$ be a $(p,q)$-loom,  and for every $i \in V(\L)=[n]$, let $\L_i=(A_i,B_i)$ be an $(r_i,s_i)$-loom. Let  $(C,D)=\L[\L_1,\dots,\L_n]$, and suppose it is a $(c,d)$-loom.  If $s_{i_1}=\cdots =s_{i_p}$ for $\{i_1,\cdots, i_p\}\in A$ and $\nu(A)=q$, then  $\nu(C)=d$ if and only if $\nu(A_i)=s_i$ for each $i$. 
\end{theorem}
\begin{proof}
For the ``if" part, let $a_1,\dots, a_q$ be a matching in $A$ and $e_{i,1}, \dots, e_{i, s_i}$ be a  matching in $A_i$. Then $\cup_{i\in a_\ell}e_{i,j}$ for $j=1,\dots,s_i$, $\ell=1,\dots,q$ is a matching of size $d$ in $C$. 

For the ``only if" part, by~\refL{lemma:loompm}, assume that $a_1,\dots, a_d$ is a perfect matching in $C$. Then $a_1\cap V(A_i),\dots , a_d\cap V(A_i)$ form a perfect matching of $A_i$ (ignoring sets  $a_j\cap V(A_i)$ that are empty),  
so by \refL{lemma:loompm} their matching number is $s_i$. 
\end{proof}

\begin{corollary}
    Let $\L=\L_1\cplus_1\L_2=(C,D)$ for $\L_1=(A_1,B_1)$ and $\L_2=(A_2,B_2)$. Then $C$ has a perfect matching if and only if $A_1$ and $A_2$ both have  perfect matchings. Also, $D$ has a perfect matching if and only if both $B_1$ and $B_2$  have a perfect matching. 
\end{corollary}

Next we claim that having a perfect fractional matching is preserved by the blow-up operation. 
\begin{theorem}\label{thm:blowuppfm}
 Given a $(c,d)$-loom $(C,D)=\L[\L_1,\dots,\L_n]$ for $\L=(A,B),\L_i=(A_i,B_i)$, where $(A,B)$ is $(p,q)$-loom and $\L_i$ are $(r_i,s_i)$-looms. If $s_{i_1}=\cdots =s_{i_p}$ for $\{i_1,\cdots, i_p\}\in A$, and
  $\nu^*(A)=q, \nu^*(A_i)=s_i$ for each $i$ then $\nu^*(C)=d$.
\end{theorem}
\begin{proof}
By~\refL{lemma:loompfm}, there exists
  a perfect fractional matching $f:A\rightarrow\R_{\ge 0}$ of~$A$, and $w:A_1\cup\cdots\cup A_n\rightarrow \R_{\ge 0}$  a perfect fractional matching of each $A_i$. Note that 
$w_i:=\sum_{e\in A_i}w(e)=s_i$. 
For each $e=e_{i_1}\cup \cdots \cup e_{i_p}\in C$, where $a=\{i_1,\dots, i_p\}\in A$ and $e_{i_\ell}\in A_{i_\ell}$, let $g(e):=f(a)\prod_{\ell=1}^pw_{i_\ell}/s(a)^{p-1}$, where $s(a):=s_{i_1}=\cdots =s_{i_p}$. We claim that $g$ is a perfect fractional matching of $C$, which by~\refL{lemma:loompfm} implies $\nu^*(C)=d$.

It is enough to show that for every $v\in V(C)$, $\sum_{e\in C:v\in e}g(e)=1$. Indeed, assume $v\in V(A_{i})$. Then
\begin{align*}
    &\sum_{e\in C:v\in e}g(e)\\
    =&\sum_{a=\{i_1,\cdots, i_p\}\in A:i\in a}\sum_{e_1\in A_{i}:v\in e_1}\sum_{e_2\in A_{i_2}}\cdots\sum_{e_p\in A_{i_p}}g(e_1\cup\cdots\cup e_p)\\
    =&\sum_{a=\{i_1,\cdots, i_p\}\in A:i\in a}\sum_{e_1\in A_{i}:v\in e_1}\sum_{e_2\in A_{i_2}}\cdots\sum_{e_p\in A_{i_p}}f(\{i_1,\dots, i_p\})w_1\cdots w_p/s(a)^{p-1}\\
    =&\sum_{a=\{i_1,\dots, i_p\}\in A:i\in a}f(\{i_1,\dots, i_p\})\sum_{e_1\in A_{i}:v\in e_1}w_1\sum_{e_2\in A_{i_2}}w_2\cdots \sum_{e_p\in A_{i_p}}w_p/s(a)^{p-1}\\
    =& 1\cdot 1\cdot s(a)^{p-1}/s(a)^{p-1}= 1,
\end{align*}
    which completes the proof. 
\end{proof}

\begin{corollary}
    If $\L,\L_1,\cdots,\L_n$ satisfy the assumptions of~\refT{thm:blowuppfm} and~\refCon{conjmain1}, then $\L[\L_1,\dots, \L_n]$ satisfies~\refCon{conjmain1}. In particular, if the $(r_1,s)$-loom $\L_1$ and the $(r_2,s)$-loom $L_2$ satisfy~\refCon{conjmain1}, so does $\L_1\cplus_1\L_2$.
\end{corollary}

\section{Stars versus perfect matchings}

For a graph $G$ let $PM(G)$ be the set of perfect matchings in the graph $G$ and $ST(G)=\{ star_G(v) \mid v\in V(G)  \}$ - both hypergraphs having $E(G)$ as their ground set.   For an $s$-regular graph~$G$ on an even number $n$ of vertices, let $r=\frac{n}{2}$. Then $A=PM(G)$ and $B=ST(G)$ are $r$-uniform and $s$-uniform, respectively. They are orthogonal, and $A=C_r(B)$. Let  $\L(G)=(A,B)$. Note that despite the notation, this is not necessarily a loom, because possibly    $B\neq C_s(A)$. 
\begin{examples}\hfill \label{examples_pm_vs_stars}
    \begin{enumerate}
        \item For $G=K_n$ with even $n\ge 6$, the pair $(A,B)$ as above is a loom. (See~\refT{thm:Knloom} below.)
\item If $G=K_{n,n}$, then $\L(G)$ is isomorphic to $\L_{n,n}$ from ~\refE{ex:looms}~\eqref{eq:a33loom}.
\end{enumerate}

\end{examples}

\begin{theorem}\label{thm:Knloom}
For $n\ge 6$ even,     $\L(K_n)$ is an $(\frac{n}{2}, n-1)$-loom.  
\end{theorem}
\begin{proof}
Let $A=PM(K_n)$ and $B=ST(K_n)$.
It is routine to check that $A$ and $B$ are orthogonal and $C_{\frac{n}{2}}(B)=A$. It remains to show that $\tau(A)=n-1$ and $C_{n-1}(A)=B$.
By Tutte's theorem  it suffices to show that after removing a non-star set $\Gamma$ of $n-1$ edges, there is no $S\subseteq V(G)$ such that the number $t$ of odd components of $G-\Gamma$ in $V\setminus S$ satisfies $t>|S|$. 

Suppose there exist such $\Gamma$ and $S$ violating this claim. Let $s=|S|$. The even-ness of $n$ and the inequality $t>s$ imply that
\[t\ge s+2.\]
Assume first that $S = \emptyset$. Then $t\ge 2$, and if one of the components of $G-\Gamma$ is of size~$p$, then $|\Gamma|\ge p(n-p)$. The assumption that $\Gamma$ is not a star implies that $p(n-p)>n-1=|\Gamma|$  (here we used the condition $n\ge 6$), a contradiction.

Assume next that $S \neq \emptyset$.
Let $1\le c_1\le \dots\le c_t$ be the sizes of the odd components of $G-\Gamma$ in $V\setminus S$. Since   $s+\sum_{i=1}^tc_i\le n$, we have $s+2\le t\le n-s$ and then
\[ s\le \frac{n-2}{2}. \]
The number of removed edges is at least $\binom{n-s}{2}-\sum_{i=1}^t\binom{c_i}{2}$. Note that
$\sum_{i=1}^tc_i\le n-s$ and $c_i\ge 1$, we have $\sum_{i=1}^t\binom{c_i}{2}\le \binom{n-s-(t-1)}{2}$ (which can easily be prove by induction on $t$) so that
\begin{align*}
    n-1&\ge \binom{n-s}{2}-\sum_{i=1}^t\binom{c_i}{2}\ge \binom{n-s}{2}-\binom{n-s-(t-1)}{2}\\
       &=\frac{1}{2}(t-1)(2n-2s-t).
\end{align*}

Note that $\frac{1}{2}(t-1)(2n-2s-t)$ attains minimum $-\frac{3}{2}s^2+(n-\frac{5}{2})s+n-1$ at $t=s+2$ for $s+2\le t\le n-s$. Therefore we should have 
\[   n-1\ge -\frac{3}{2}s^2+(n-\frac{5}{2})s+n-1. \]
Then the minimum of the right-hand side is
\[ \min\Big(2n-5,\quad \frac{(n-2)(n-4)}{8}+n-1\Big) \]
at $s=1$ or $\frac{n-2}{2}$. But it is impossible to have
\[ n-1\ge \min\Big(2n-5,\quad  \frac{(n-2)(n-4)}{8}+n-1\Big),  \]
assuming $n\ge 6$, which is a contradiction.
\end{proof}

If $G$ is an $s$-regular graph  on $n$ vertices and $ST(G)\subsetneqq C_s(PM(G))$, namely $\L(G)$ is not a loom, we may hope that the pair  
 $(PM(G), C_s(PM(G)))$ is a loom. For this to be true,  we need that  $C_s(PM(G)) \perp PM(G)$. This sometimes happens and sometimes not, as shown by the following two examples. 

\begin{examples}\hfill
\begin{enumerate}
   
    \item Let $P$ be the Petersen graph, $A=PM(P), B=C_3(PM(P))$. The hypergraph $B$ includes, besides the stars, five triples that are themselves perfect matchings. 
In the drawing of Petersen's graph, such a triple consists of two parallel edges, one on the inner pentagon and one on the outer, and an edge perpendicular to both, connecting the two pentagons. Then $(A,B)$ is a $(5,3)$-loom in which $B$ has a perfect matching, and $A$ does not (since the chromatic index of the Petersen graph is 4).  
\item 
Let $G$ be the $3$-regular graph obtained from Petersen's graph, by blowing up (in the graph theory sense. Not to confuse with the blow-ups defined in~\refS{subsec:genblowup}) every vertex by a triangle. The set of $3$ edges incident with every triangle (=blown-up vertex) is obviously a cover for $A$, and on the other hand, it is contained in a perfect matching of the blow-up graph.
    \end{enumerate}
\end{examples}
    
The following theorem states that any loom of the form $(PM(G), C_s(PM(G)))$ for some $s$-regular graph $G$ satisfies \refCon{conjmain1}. 
\begin{theorem}\label{thm:LGpfm}
    For any $s$-regular graph, if $\Big(PM(G), C_s(PM(G))\Big)$ is an $(r,s)$-loom, then both its components, $PM(G)$ and $C_s(PM(G))$, have perfect fractional matchings.
\end{theorem}
 $ST(G)$ has a perfect fractional matching, obtained by putting weight $\frac{1}{2}$ on every star. Since $C_s(PM(G)) \supseteq ST(G)$, this is also a perfect fractional matching of~$C_s(PM(G))$.  
 
 The existence of a perfect fractional matching of $PM(G)$ is given by the following theorem. 
\begin{theorem}\label{thm:pmvsstar}
  Let $G$ be an $s$-regular graph on $n$ vertices and let $A=PM(G)$. If $\tau(A) \ge s-1$ then  $\nu^*(A)=s$.  
\end{theorem}
  The fractional edge-chromatic number $\chi_e^*(G)$ (sometimes denoted by $\chi_f'(G)$) of a graph $G$ is  the minimum total weight of a fractional edge-coloring of $G$, i.e., $\min \sum_{M\in\mathcal{M}(G)}f(M)$ over all non-negative real functions $f$ on $\M(G)$ satisfying
  \[\sum_{M\in\mathcal{M}(G):e\in M}f(M)\ge 1\] 
  for every $e\in E(G)$.

Given  $U \subseteq V(G)$, let $t(U) = \frac{2|E(G[U])|}{|U|-1}$. Let 
\[t(G)= \max_{U\subseteq V(G):\text{$|U|$ is odd}} t(U).\]

We use the following result~\cite[Theorem 4.2.1]{scheinerman2011fractional}.

\begin{lemma}\label{lemma:edgefractional}
    \[\chi_e^*(G) = max(\Delta(G), t(G)).\]
\end{lemma}

\begin{proof}[Proof of~\refT{thm:pmvsstar}]
To prove $\nu^*(A)=s$, we claim that it is enough to show that $\chi_e^*(G)=s$.
Indeed,  let $f: \mathcal{M}(G) \to [0,1]$ be a fractional edge coloring of $G$ with $\sum_{M\in \mathcal{M}(G)}f(M)=s$. We may assume that $\sum_{M\in\mathcal{M(G)}:e\in M}f(M)=1$ for any $e\in E(G)$ as $\mathcal{M}(G)$ is a simplicial complex.
Then 
\begin{align*}
   |V(G)|s &= \sum_{v\in V(G)}\sum_{e\in E(G):v\in e}\sum_{M\in\mathcal{M}:e\in M}f(M)\\
    &=\sum_{M\in\mathcal{M}}\sum_{e:e\in M}\sum_{v:v\in e}f(M)=\sum_{M\in\mathcal{M}}2|M|f(M)
\end{align*}
Note that $|M|\le |V(G)|/2$, and $\sum_{M\in \mathcal{M}}f(M)=s$. This implies that $M$ is a perfect matching of $G$ whenever $f(M)>0$. Therefore this $f$ forms a fractional matching of $A$ of total weight~$s$.

To prove $\chi_e^*(G)=s$, note that for any $U \subseteq V(G)$ of odd size, $\emptyset\subsetneqq U\subsetneqq V(G)$ as $|V(G)|$ is even, therefore the cut $E(U,V(G)\setminus U)$ is a cover of $A$. Since $|E(U,V(G)\setminus U)|=s|U|-2|E(G[U])|$ has different parity from $s-1$ and since $\tau(A)\ge s-1$,  we have $|E(U,V(G)\setminus U)| \ge s$. Then $|E(G[U])|\le \frac{s|U|-s}{2}$, implying 
\[t(U) = \frac{2|E(G[U])|}{|U|-1} \le \frac{2\frac{s(|U| -1)}{2}}{|U|-1} =s.\] 
Since~$\Delta(G)=s$,
then \refL{lemma:edgefractional} implies that $\chi_e^*(G)=\max(\Delta(G), t(G))=s$. 
\end{proof}

This concludes the proof of~\refT{thm:pmvsstar}, and hence also of~\refT{thm:LGpfm}.

So far, in all our examples at least one component of the looms has a perfect (integral) matching. 
    In the following example, neither component of the loom has a perfect matching.   

\begin{example}
      Let $\L_1=(A_1,B_1)=(ST(K_{10}),PM(K_{10}))$, which is a $(9,5)$-loom by~\refT{thm:Knloom} and $\nu(A_1)=1<5$. Let $\L_2=(A_2,B_2)=(PM(K_6),ST(K_6))$, which is a $(3,5)$-loom by~\refT{thm:Knloom} and $\nu(B_2)=1<3$. Then $\L=(A,B)=\L_1\cplus_1\L_2$ is a $(12,5)$-loom with no perfect matching in either component by~\refT{thm:pmiffpmineach}. 
\end{example}

\section{$(r,2)$-looms }

     An $(r,1)$-loom is just $\V_r$, from \refE{ex:looms} (2), meaning that if it is non-trivial then it is decomposable. As we shall see, this is not true for $(r,s)$-looms when $r,s\ge 3$, but it is true for $s=2$.

\begin{theorem}\label{char_r2_looms}
   Any $(r,2)$-loom $\L=(A,B)$ is decomposable.  
\end{theorem}
By \refL{lemma:decomposablenotconnected} it suffices to show that at least one of $A, B$ is disconnected. This will follow from:  

\begin{lemma}\label{lemma:compbipartite}
   In an $(r,2)$-loom $\L=(A,B)$,  every connected component $D$ of $B$ is a complete bipartite graph with two sides of the same size. 
\end{lemma}

This already proves the theorem: in an $(r,2)$-loom $\L=(A,B)$, if $B$ is connected, then $A$, which consists of the two sides of the complete bipartite graph, is disconnected.

\begin{proof}[Proof of \refL{lemma:compbipartite}] 
As to bipartiteness, suppose, for contradiction, that $B$ contains an odd cycle $O$. Then any edge of $A$, being a cover of $B$, must contain both endpoints of some edge of $O$, contradicting the fact that $A \perp B$. 
To prove that $D$ is complete, it suffices to show that if $N_B(x) \cap N_B(y) \neq \emptyset$ then $N_B(x) = N_B(y)$. So, assume for negation that $u \in N_B(x) \cap N_B(y)$ and $v \in N_B(x) \setminus N_B(y)$. Then any $a \in A$ not containing $y$ must contain $u$, hence it avoids $x$, and to cover $xv$ it must contain $v$. Thus $yv$  is a cover of $A$, and being minimum it must belong to $B$, as desired. 
Therefore each connected component of $B$ is a complete bipartite graph.

Let $D_i,~1\le i \le m$ be the connected components of $B$, and let $X_{i,j},~~j =1,2$ be the sides of $D_i$, where $|X_{i,1}| \le |X_{i,2}|$. Then $\tau(B)= \sum_{1\le i \le m}|X_{i,1}|$. Remember that $\tau(B)=r$.

We claim that
\[|X_{i,2}|= |X_{i,1}|\] for all $1\le i \le  m$.
To prove the claim, assume, for contradiction, that 
$|X_{i,2}|> |X_{i,1}|$ for some $i$. Let $a$ be an edge of $A$ meeting 
    $X_{i,2}$. By the orthogonality condition $a \cap X_{i,1}=\emptyset$, and since it is a cover of $B$, it contains   $X_{i,2}$. For the same reason $a$ contains a side from each $D_i$, so $|a|> r$, contradicting the fact that $A$ is $r$-uniform.
This  completes the proof of the lemma.
\end{proof}

The above discussion provides a characterization of all $(r,2)$-looms:
 
\begin{theorem}\label{thm:chars2}
For an $(r,2)$-loom $(A,B)$, there exist pairs of sets $(X_{i,1},X_{i,2})_{1\le i \le t}$ satisfying:
   \begin{itemize}
       \item All $X_{i,j}$ are pairwise disjoint, $|X_{i,1}|=|X_{i,2}|=:q_i$,
       \item $\sum_{1\le i\le t}q_i=r$, 
        \item  $A=\{\bigcup X_{i,\sigma(i)} \mid \sigma \in [2]^{[t]}\}$ and $B=\bigcup_{1\le i \le t}\{uv\mid u\in X_{i,1}, v\in X_{i,2} \}$.     
   \end{itemize}
\end{theorem}

Here $[2]^{[t]}$ is the  set of functions from $[t]$ to $[2]$.

Let    
        \begin{align*}
            \L_i &=\Big(\{X_{i,1},X_{i,2}\}, \{uv\mid u\in X_{i,1}, v\in X_{i,2} \}\Big)\\
                 &=\Big(\{X_{i,1}\},\big\{\{x\}\big\}_{x\in X_{i,1}}\Big)\cplus_2\Big(\{X_{i,2}\}, \big\{\{y\}\big\}_{y\in X_{i,2}}\Big)\\
                 &=\V_{q_i}\cplus_2 \V_{q_i}.
        \end{align*}

                 Then $\L=\L_1\cplus_1\cdots \cplus_1 \L_t=(\V_{q_1}\cplus_2\V_{q_1})\cplus_1\cdots \cplus_1 (\V_{q_t}\cplus_2\V_{q_t})$.

    \refT{thm:chars2} implies that $\nu(A)=2$ and $\nu(B)=r$, thus $\nu^*(A)=2$ and $\nu^*(B)=r$. Hence~\refCon{conjmain1} is true when $s=2$.

\section{$(3,3)$-looms}

\begin{theorem}\label{thm:33loom}
    Let  $\L=(A,B)$ be a $(3,3)$-loom on the vertex set $V$. 
    Then 
    \begin{enumerate}
        \item\label{eq:V=9} $|V|=9$, 
        \item\label{eq:nu=3} $\nu(A)=\nu(B)=3$, 
        \item\label{eq:Vminuesef} For any pair $e,f$ of disjoint edges in $A$, $V\setminus (e\cup f)\in A$.
    \end{enumerate}
    \end{theorem}
\begin{proof} 
\begin{claim}\label{claim:equivalence}
    The three conditions are equivalent. 
    \end{claim}
\begin{proof}[Proof of the claim]

The implication \eqref{eq:Vminuesef} $\Rightarrow$  \eqref{eq:nu=3} follows from~\refL{lemma:nu2}.

 \eqref{eq:V=9} $\Rightarrow$ ~\eqref{eq:Vminuesef}: suppose $|V|=9$. For any two disjoint edges $e,f\in A$ and any $g\in B$, we have $|g\cap e|=|g\cap f|=|g\cap (V\setminus(e\cup f))|=1$, which implies $V\setminus(e\cup f)\in C_3(B)=A$ so that $\nu(A)=3$.

 \eqref{eq:nu=3} $\Rightarrow$  \eqref{eq:V=9}: let $e,f,g$ be three disjoint edges in $A$. If there exists a vertex $v \not \in e \cup f\cup g$, then an edge $h \in B$ containing $v$  cannot cover $e, f, g$, contradicting the assumption that $B=C_3(A)$. 
 \end{proof}

   Back to the proof of the theorem, let $k=\max_{e_1,e_2,e_3\in A}|\cup_{i=1}^3e_i|$.~\eqref{eq:nu=3} follows from the next result. 
\begin{claim}\label{claim:k9}
    $k=9$.
\end{claim}
\begin{proof}[Proof of the claim]
By \refL{lemma:nu2} there are two disjoint edges, $e_1,e_2$ belonging to $A$. By the orthogonality condition, no edge of $B$ is contained in $e_1 \cup e_2$, so 
$V\setminus(e_1\cup e_2)$ is a cover for $B$, and hence it is of size at least $\tau(B)=3$. This shows that $|V|\ge 9$.
Adding to $e_1,e_2$ any edge meeting $V\setminus(e_1\cup e_2)$  proves $k>6$. 

Next we exclude the possibility  $k=8$. 
$k=8$ means that there exist edges  $e_1,e_2,e_3\in A$  such  that $e_1\cap e_2=e_2\cap e_3=\emptyset$ and $|e_1\cap e_3|=1$, say $e_1\cap e_3=\{x\}$. Take $y\in V\setminus (e_1\cup e_2\cup e_3)$. Then $star_B(y)\subseteq star_B(x)$, since for an edge $b\in B$ containing $y$ to meet all edges $e_1, e_2, e_3$ it must contain $x$.  By \refL{lemma:star=star}  $star_B(x)=star_B(y)$. Since $e_3$ is a cover of $B$, we have that $(e_3\setminus\{x\})\cup\{y\}$ is also a cover of $B$, so it is in $A$. Together with  $e_1$ and $e_2$ it contains $9$ vertices, contrary to the assumption that $k=8$.

It remains to exclude the case $k= 7$. Assuming for negation that this is the case,  since $\nu(A),\nu(B)\ge 2$, we can number the vertices so that $123,456\in A$, $714,825\in B$. Let $e$ be an edge of $A$ containing $9$. Since $k=7$, the edge $e$ contains vertices $x\in \{2,5\}$ and $y\in \{1,4\}$. Since $star_A(x)\neq star_A(9)$ and $star_A(y)\neq star_A(9)$, which implies $star_A(9)	\nsubseteq star_A(x)$ and $star_A(9)\nsubseteq star_A(y)$ by~\refL{lemma:star=star} so that $\{9\}\cup \{2,5\}\setminus\{x\}$ and $\{9\}\cup\{1,4\}\setminus\{y\}$ are each contained in some edge of~$A$. On the other hand,  any $f\in B$ containing $9$ must intersect the edges $123$ and $456$ of~$A$, but such~$f$ cannot include any element of $\{1,2,4,5\}$ by the orthogonality assumption. Therefore $936\in B$ and we have $\nu(B)=3$ and $|V|=k=9$, a contradiction.

This completes the proof of \refCl{claim:k9}. 
\end{proof}

By \refCl{claim:equivalence} this also proves the theorem.
\end{proof}

     In particular, this entails that in a $(3,3)$-loom $(A,B)$ there holds $\nu^*(A)=
     \nu(A)=\nu^*(B)=\nu(B)=3$.  

\begin{corollary}
    \refCon{conjmain1} is true for $r=s=3$, and hence \refCon{conj:GL} is true for $r=3$.
\end{corollary}

     The next theorem characterizes   $(3,3)$-looms. If a $(3,3)$-loom $\L=(A,B)$ is decomposable, then 
     $(A,B)=(A_1,B_1)\cplus_1 (A_2,B_2)$ for a $(2,3)$-loom $(A_1,B_1)$ and a $(1,3)$-loom $(A_2,B_2)$, whose structure we know by~\refT{thm:chars2}. So we only need to know how indecomposable $(3,3)$-looms look. It turns out that there are just two such looms.

\begin{theorem}\label{thm:33loomchar}
    The only indecomposable $(3,3)$-looms are~$\L_{3,3}$ in~\refE{ex:looms}~\eqref{eq:a33loom} and~$\V_{3,3}$ in \refE{ex:vane}.
\end{theorem}
\begin{proof}
    Let $\L=(A,B)$ be an indecomposable $(3,3)$-loom. Let $M,N$ be perfect matchings in $A$ and in $B$, respectively. 
     Number the vertices so that  ``columns" $M=\{147,258,369\}$ and $N=\{123,456,789\}$ 
     and view them as the sets of rows (resp. columns) of a $3 \times 3$ grid. Since $A$ is connected, without loss of generality, we may assume that at least one of the edges of  $159$ or $158$ is in $A$.

\begin{enumerate}[(I)]
    \item $159\in A$.  By~\refL{lemma:nu2}, there exists an edge of $A$ disjoint from $159$.
\begin{enumerate}[(1)]
    \item Suppose there exists in $A$ another ``permutation" edge (whose vertices are in distinct columns and rows), say $267$. Then by~\refL{thm:33loom},  $[9]\setminus (159\cup 267)=348$ is in $A$.
    By renaming the vertices (swap 1 with 7, 5 with 8, and 3 with 6), the current known six edges of $A$ become ``rows" and ``columns" and the three edges of $B$ become ``even permutation" edges. Since $B$ is also connected and the edges of $B$ are orthogonal to those in $A$, then the remaining edges of $B$ must be ``odd permutation" edges. By~\refL{lemma:nu2} three odd permutation edges are in $B$. Therefore, in this case, the loom is isomorphic to $\L_{3,3}$ in~\refE{ex:looms}~\eqref{eq:a33loom}.

    \item  Suppose there is no permutation edge of $A$ other than $159$. By~\refL{lemma:nu2}, there exists an edge of $A$ disjoint from $159$. Without loss of generality,  $247 \in A$.
Then $368=[9]\setminus (159\cup 247)\in A$. Also $158=[9]\setminus (247\cup 369)\in A$ and $259=[9]\setminus(147\cup 368)\in A$. We now have eight edges in $A$,  and we turn to $B$. Since $B$ is connected, by the orthogonality of $A$ and $B$  and by the covering property, the only possible edges to connect $123$ with $456$ or $789$ are $126$, or $345$, or $357$.
\begin{enumerate}[(i)]
    \item If $357\in B$, then by~\refL{lemma:nu2} and~\refT{thm:33loom}, there is an edge $e$ of $B$ containing 6 and being disjoint from $357$. By the orthogonality, this edge cannot contain $8$ or $9$ as $369,368\in A$. Since $e$  is disjoint from $357$ it does not contain $5$, so to cover $259\in A$, it must contain $2$. By the orthogonality, $e$ cannot contain $4$ or $7$, as $247\in A$.  Since $e$ is disjoint  from  $357$, to cover $147\in A$,   it must be $126$. Then $489=[9]\setminus (357\cup 126)\in B$. This implies $345=[9]\setminus (126\cup 789)\in B$ and $756=[9]\setminus (489\cup 123)\in B$. We now have eight edges of $B$, and it is easy to check that $(A,B)$ is isomorphic to~$\V_{3,3}$ in~\refE{ex:vane}.
    \item Note that $126\cup 345\cup 789 =[9]$. If one of $126$ and $345$ is in $B$, then the other is also in $B$. The edge $789$ is connected to $123$ or $456$ in $B$, symmetrically, one of $756$ or $489$ or $357$ is in $B$. If $357$ is in~$B$, then we have done by the above case. We may assume that both $756$ and $489$ are in $B$ (as $756\cup 489\cup 123=[9]$). Then $357=[9]\setminus (126\cup 489)\in B$ and we are done.    
\end{enumerate}

\end{enumerate}
\item Suppose $158$ is in $A$ and no permutation edge (that intersects each of $147,258,369$ at one vertex) is in $A$.
    Then $247=[9]\setminus(158\cup 369)$ is in $A$. Since $A$ is connected, we may assume there is an edge of $A$ contained in $258\cup 369$.
    \begin{enumerate}[(1)]
        \item $269$ or $358$ is in $A$. Note that $147\cup 269\cup 358=[9]$. If one of them is in $A$, then so is the other. Then $169=[9]\setminus (247\cup 358)$ is in $A$ and then $347=[9]\setminus(169\cup 258)$ is in $A$. 
        Let us now go back to $B$. By assumption it is connected, and hence $123$ should be connected to $456$ or to $789$. But for any $x\in 123$ and $y\in 456\cup 789$, $xy$ is contained in some edge of $A$, contradicting the orthogonality of $A$ and $B$.
\item $259$ and $368=[9]\setminus(259\cup 147)$ are in $A$. Then $159=[9]\setminus (247\cup 368)$ is in $A$, a contradiction to the assumption that no permutation edge is in $A$.
    \end{enumerate}
\end{enumerate}
We complete the proof of the theorem.
\end{proof}

{\noindent\bf Acknowledgements.}  The authors are grateful to Zilin Jiang and  Martin Loebl for discussions on an earlier version of the paper. We thank the anonymous referees for their comments.

 \end{document}